\documentclass[12pt,a4paper]{article}
\usepackage[left=1cm, right=1cm, top=1cm, bottom=1cm]{geometry}
\usepackage{graphicx}
\usepackage{mathtools}
\usepackage{amssymb}
\usepackage{amsthm}
\usepackage{thmtools}
\newtheorem{definition}{Definition}[section]
\newtheorem{prop}{Proposition}[section]
\newtheorem{thm}{Theorem}[section]
\newtheorem{cor}{Corollary}[section]
\usepackage{multirow}
\usepackage[english]{babel}
\begin{document}
	\begin{center}
		\textbf{Enumeration of subsets with closedness in finite fields of characteristic 2}
	\end{center}
	\begin{center}
		\textbf{Nithish Kumar R$^1$, Vadiraja Bhatta G. R.$^2$ and Prasanna Poojary$^3$}\\[3pt]
		$^{1,2}$Department of Mathematics, Manipal Institute of Technology, Manipal Academy of Higher
		Education, Manipal 576104, India.\\[3pt]
		$^3$Department of Mathematics, Manipal Institute of Technology Bengaluru, Manipal Academy of
		Higher Education, Manipal 576104, India.\\[3pt]
		$^1$nithishk53@gmail.com, $^2$vadiraja.bhatta@manipal.edu, $^3$poojary.prasanna@manipal.edu
	\end{center}
	\begin{abstract}
		The additive closedness in the subset of an additive group is termed as $r$-value. The nature of closedness in different subsets of fixed size is observed as a spectrum of $r$-values. We enumerate $r$-values of subsets in finite fields of characteristic 2 and represent them as the spectrum of values. Based on these values the subsets can be further studied as partial Steiner triple systems, sum-free sets, Sidon sets, and Schure triples.
	\end{abstract}
	\section{Introduction}	
	In an additive group $G$, a subset is associated with a non-negative integer called $r$-value, representing the amount of closedness in it. This concept was initially defined and studied by Sophie et al. in 2009 \cite{HUCZYNSKA2009831}. Formally, for a subset $A$ of an additive group $G$, $r$-value of $A$ denoted by $r(A)$ is defined as the number of ordered pairs in $A$, such that for each pair sum of its elements is again an element of $A$. For prime $p$, Sophie et al. studied $r$-values of subsets of $\mathbb{Z}_p$ and presented as a spectrum of $r$-values of fixed subset size in $\mathbb{Z}_p$ \cite{HUCZYNSKA2009831}. Continuing, r-closed sets of natural number particularly $r$-values of intervals $[1, N]$ in $\mathbb{N}$ were studied by Sophie in 2014 \cite{Huczynska2011}. In 2024, Nithish et al. defined $r(A,B,C)$, where $A,B,C$ are subsets of an additive group $G$. They obtained $r$-values of subsets of $\mathbb{Z}_n$ by representing each subset as a union of intervals in $\mathbb{Z}_n$ and deriving formula for $r$-values of intervals in $\mathbb{Z}_n$.
	
	\par The zero valued subsets, commonly known as Sum-free sets can be regarded as the initial approach to study $r$-values of subsets. In a sequence of integers, properties like Sum-free sets and Sidon sets connected with $r$-values were studied even before 1965, as mentioned by P. Erdös in \cite{Erdo65}. Terence and Vu recently addressed Erdös' problems in their study on Sum-free sets in groups \cite{sumfreesurvey}. In between, Sum-free sets were studied in view of the maximum cardinality of zero valued sets [See \cite{axioms12080724}, \cite{cameron1990}],  geometric study of Sum-free sets in integer lattice grid [See \cite{LIU2023103614}, \cite{Elsholtz}]. Zero valued sets were used in the process of discovering new	APN(Almost Perfect Nonlinear) exponents \cite{Claude2023} and also results on F-saturated graphs for large complete graphs \cite{TIMMONS2022112659}.

	\par Cardinality of Sidon sets and the generalization of Sidon’s problem on Sidon set is related to r-closed set [See \cite{MARTIN2006591}, \cite{CILLERUELO20102786}]. In \cite{Mullin}, the authors introduced the concept of a relative anti-closure property for subsets to study the mathematical model of the biological concept of mutation. This anti-closure property is connected to $r$-closed when $r$ is zero. Datskovsky in \cite{DATSKOVSKY2003193} discussed $r$-values over complete residue modulo $n$ as the number of Schur triples. The minimum number of additive tuples in groups of prime order is studied in 2019 by Ostap et al. \cite{Chervak2017}.

	For a subset $A$ of $\mathbb{F}_{2^n}\backslash\{0\}$, let $\mathcal{R}$ be the set of all 3-subsets of $A$ such that in each 3-subset sum of two elements is equal to the remaining one element. If $a,b,c\in A$ with $a+b=c$ then the triplet $\{a,b,c\}$ contribute 6 to $r(A)$. If another triplet $\{d,e,f\}$ contributing to $r(A)$, then the intersection of these two triplets can be atmost a singleton set. Thus $(A,\mathcal{R})$ is a partial Steiner triple system of order $|A|$ [see \cite{steiner2020}]. This interlink motivated to explore $r$-values of subsets of $\mathbb{F}_{2^n}$. 
	
	Section \ref{preliminary} covers the basic results and observations of $r$-values over subsets of a finite additive group. Section \ref{r 2_n} illustrates the results which generate the spectrum of $r$-values of $\mathbb{F}_{2^3}$. Section \ref{Spec 2_n} presents the results for $n\geq3$, $r$-values in $\mathbb{F}_{2^{n+1}}$ is generated using $r$-values in $\mathbb{F}_{2^n}$. 
	
	\section{Preliminary}\label{preliminary}
	In this section, we will provide some definitions and fundamental results. In addition, we present a few basic results that we established.
	
	\begin{definition}\label{D2}
		\cite{Nithish} For any subsets $ A,B,C$ of a finite additive group $ G $, we define $ r(A,B,C) $ as the cardinality of the set $ \{(a,b)~|a\in A,~b\in B,~a+b\in C\} $.
	\end{definition}
	If $ A=B $, then we denote $ r(A,B,C)=r(A,C) $ and if $ A=B=C $ then $ r(A,B,C)=r(A) $, which reduce to $ r $-value of $ A $. 
	
	\par If $A$, $B$ and $C$ are three disjoint subsets of an additive abelian group $G$, using  the property $r(A,B,C)=r(B,A,C)$ we have $r(A\cup B)=r(A)+r(A,B)+2r(A,B,A)+2r(A,B,B)+r(B,A)+r(B)$.
	
	In the above definition, we call the set associated with $ r(A,B,C) $ as $R$-set of $A,B,C$. 
	\begin{definition}\label{D1}
		For any subsets $ A,B,C$ of an additive group $ G $, we define $R$-set of $A,B,C$ denoted by $R(A,B,C)$ as the set $\{(a,b,c)\in A\times B\times C~|~a+b=c\} $.
	\end{definition}
	Note that $ r(A,B,C) =|R(A,B,C)|$. The following result gives a bound of $ r(A,B,C)$.
	
	\begin{prop}
		Let $A$, $B$ and $C$ are three subsets of a finite additive abelian group $G$ with $|A|\leq|B|$. Then $r(A,B,C)\leq|A|\cdot min\{|B|,|C|\}$.
	\end{prop}
	\begin{proof}
		For a fixed $a$ in $A$, there are $|B|$ number of pairs $(a,b)$ with $b\in B$ giving $|B|$ distinct elements $a+b$ of $G$. Suppose $|C|\leq|B|$, then among $|B|$ distinct elements atmost $|C|$ elements lies in $C$. Therefore, $r(\{a\},B,C)\leq|C|$ and hence $r(A,B,C)\leq|A||C|$. Similarly suppose $|B|<|C|$, we have $r(A,B,C)\leq|A||B|$.
	\end{proof}
	
	The following result connect the $r$-value of a set and its compliment in an additive abelian group $G$.
	
	\begin{thm}\label{compliment}
		\cite{HUCZYNSKA2009831} Let $ G $ be a finite abelian group of order $ g $. Let $ k $ be a positive integer with $ 0\leq k\leq g $, and let $ A $
		be a subset of $ G $ of size $k $. Let $ \overline{A} $ be the complement of $ A $ in $ G $. Then $ r(A)+r(\overline{A})=g^2-3gk+3k^2 $.
	\end{thm} 
	The proof of this result follows from the proof of 4 in Corollary \ref{cor comp 3}. The following Theorem and Corollary connect the relation of $r$-values of sets and its compliment with respect to the Definition \ref{D2}. The relation 1 of Theorem \ref{comp in 3} gives the partition of $A\times B$ with the sum of each pair either lies in $C$ or not with respect to the set $R(A,B,C)$ and its $r$-values. In relation 2, $G\times B$ is partitioned into two sets, in which for each pair with sum in $C$ either first element in $A$ or not.
	
	\begin{thm}\label{comp in 3}
		Let $A$, $B$, and $C$ are three subsets of a finite additive abelian group $G$. Let $\overline{A}$ and $\overline{C}$ are compliments of $A$ and $C$ respectively. Then
		\begin{enumerate}
			\item $r(A,B,C)+r(A,B,\overline{C})=|A||B|$.
			\item $r(A,B,C)+r(\overline{A},B,C)=|B||C|$.
		\end{enumerate}
	\end{thm}
	\begin{proof}
		Given $C$  and $\overline{C}$ are compliments, we have $R(A,B,C)\cap R(A,B,\overline{C})=\emptyset$ and $R(A,B,C)\cup R(A,B,\overline{C})=R(A,B,G)$. Therefore, $r(A,B,C)+r(A,B,\overline{C})=|R(A,B,C)|+|R(A,B,\overline{C})|=|R(A,B,G)|=r(A,B,G) =|A||B|$. Similarly for $A$ and its compliment $\overline{A}$ we get $r(A,B,C)+r(\overline{A},B,C)=|B||C|$.
	\end{proof}
	\begin{cor}\label{cor comp 3}
		Let $A$, $B$, and $C$ are three subsets of a finite additive abelian group $G$. Let $\overline{A}$, $\overline{B}$ and $\overline{C}$ are compliments of $A$, $B$ and $C$ in $G$ respectively. Then
		\begin{enumerate}
			\item $r(A,B,C)+r(A,\overline{B},C)=|A||C|$.
			\item $r(A,B,C)-r(\overline{A},B,\overline{C})=(|A|-|\overline{C}|)|B|$.
			\item $r(A,B,C)-r(\overline{A},\overline{B},C)=(|B|-|\overline{A}|)|C|$.
			\item $r(A,B,C)+r(\overline{A},\overline{B},\overline{C})=|B||C|-|\overline{A}||C|+|\overline{A}||\overline{B}|$.
			
		\end{enumerate}
	\end{cor}
	\begin{proof}
		\begin{enumerate}
			\item We have  $r(A,B,C)+r(A,\overline{B},C)=r(B,A,C)+r(\overline{B},A,C)$. Using 2 of Theorem \ref{comp in 3}, we have $r(B,A,C)+r(\overline{B},A,C)=|A||C|$.
			\item Using 2 of Theorem \ref{comp in 3} in $r(A,B,\overline{C})$, we have $r(A,B,\overline{C})+r(\overline{A},B,\overline{C})=|B||\overline{C}|$. Now combining $r(A,B,\overline{C})+r(\overline{A},B,\overline{C})=|B||\overline{C}|$ and 1 of Theorem \ref{comp in 3}, we get $r(A,B,C)-r(\overline{A},B,\overline{C})=(|A|-|\overline{C}|)|B|$.
			\item Similarly follows as above.
			\item Using 1 of Theorem \ref{comp in 3}, we have $r(\overline{A},\overline{B},C)+r(\overline{A},\overline{B},\overline{C})=|\overline{A}||\overline{B}|$. Substituting for $r(\overline{A},\overline{B},C)$ from 3, we get $r(A,B,C)+r(\overline{A},\overline{B},\overline{C})=|B||C|-|\overline{A}||C|+|\overline{A}||\overline{B}|$.
		\end{enumerate}
	\end{proof}
	Note that the relation in 4 of Corollary \ref{cor comp 3} is further reduced to $r(A,B,C)+r(\overline{A},\overline{B},\overline{C})=|G|^2-(|A|+|B|+|C|)|G|+|A||B|+|A||C|+|B||C|$. 
	
	As mentioned by Sophie et al. in \cite{HUCZYNSKA2009831}, empty sets are considered zero valued sets.  Singleton sets and subgroups have $r$-values as follows.
	
	\begin{prop}\label{size 1}
		\cite{HUCZYNSKA2009831} Let $A$ be a subset of a finite abelian group with $|A|=k$.
		\begin{enumerate}
			\item If $k=1$, then $r(A)=\begin{cases}
				1 & \text{if } A=\{0\},\\
				0& \text{otherwise.}
			\end{cases}$
			\item   $A$ is a subgroup of $G$ if and only if $r(A)=k^2$.
		\end{enumerate}
	\end{prop}
	
	
	\section{$r$-values of subsets of $\mathbb{F}_{2^n}$}\label{r 2_n}
	Consider a finite field $\mathbb{F}_{2^{n-1}}$, where $n$ is a positive integer. Let $[\alpha_1,\alpha_2,\dots,\alpha_{n-1}]$ and $[\beta_1,\beta_2,\dots,\beta_{n-1}]$ be two elements of $\mathbb{F}_{2^{n-1}}$, where each $\alpha_i,\beta_i\in\{0,1\}$. If $[\alpha_1,\alpha_2,\dots,\alpha_{n-1}]+[\beta_1,\beta_2,\dots,\beta_{n-1}]=[\gamma_1,\gamma_2,\dots,\gamma_{n-1}] $ in $\mathbb{F}_{2^{n-1}}$, then $[0, \alpha_1,\alpha_2,\dots,\alpha_{n-1}]+[0,\beta_1,\beta_2,\dots,\beta_{n-1}]=[0,\gamma_1,\gamma_2,\dots,\gamma_{n-1}] $ in $\mathbb{F}_{2^n}$. Using this, for a subset $A$ of $\mathbb{F}_{2^{n-1}}$, if every element of $A$ is prefixed with $0$ and calculate its $r$-value in $\mathbb{F}_{2^n}$ is same as $r$-value of $A$ in $\mathbb{F}_{2^{n-1}}$. Thus throughout the discussion we treat $\mathbb{F}_{2^{n-1}}$ as subset of $\mathbb{F}_{2^n}$.
	
	
	Let $A$ be subset of $\mathbb{F}_{2^n}$, where $n$ is a positive integer.
	The following result gives the lower bound of $r$-values for a subset of $\mathbb{F}_{2^n}$ which contains zero element. 
	\begin{prop}\label{bound}
		If $A\subseteq \mathbb{F}_{2^n}$ and $0\in A$, then $r(A)\geq 3|A|-2$.
		
	\end{prop}
	\begin{proof}
		For all $a\in A$, we have $a+0=a$, $0+a=a$ and $a+a=0$. Hence contributing $3|A|-2$ pairs to $r(A)$.
	\end{proof}
	The proposition \ref{size 2} and \ref{3_subsset} gives the spectrum of $r$-values of subsets of size 2 and 3 in $\mathbb{F}_{2^n}$.
	\begin{prop}\label{size 2}
		If $A\subseteq \mathbb{F}_{2^n}$ with $|A|=2$, say $A=\{a,b\}$ then $r(A)=\begin{cases}
			0 & \text{if } a\neq 0\text{ and }b\neq 0\\
			4& \text{otherwise.}
		\end{cases}$
	\end{prop}
	\begin{proof}
		Suppose $a\neq 0$ and $b\neq 0$. We have $a+a=0$ and $b+b=0$, but then $0\not\in A$. Also $a+b=a$ implies $b=0$, which is not possible. similarly $a+b=b$ is not possible and hence $a+b\not\in A$.\\
		Suppose $a=0$, then by Proposition \ref{bound}, we have $3\cdot2-2\leq r(A)\leq 4$ giving $r(A)=4$.
	\end{proof}
	\begin{prop}\label{3_subsset}
		If $A\subseteq \mathbb{F}_{2^n}$ with $|A|=3$, then $r(A)$ belongs to $\{0,6,7\}$.
	\end{prop}
	\begin{proof}
		Suppose $0\in A$, say $A=\{0,b,c\}$. Note that if $b+c=0$, then $b=c$, which is not possible. Similarly if $b+c=b$ or $b+c=c$ then $c=0$ or $b=0$. Thus $b+c\notin A$, hence in this case $r(A)=7$.
		
		Suppose $0\notin A$. Say $A=\{a,b,c\}$. Using the Cayley table below,\\
		\begin{center}
			\begin{tabular}{c||c|c|c}
				& a & b & c \\
				\hline
				\hline
				a & 0 & a+b & a+c \\
				\hline
				b & b+a & 0 & b+c \\
				\hline
				c & c+a & c+b & 0 \\
			\end{tabular}
		\end{center}
		
		if $a+b$, $a+c$ and $b+c$ are not in $A$, then $r(A)=0$. Suppose $a+b\in A$, then by above arguments $a+b\neq a$ and $a+b\neq b$. But then $a+b=c$, which also implies $a+c=b$ and $b+c=a$. Hence $r(A)=6$.
	\end{proof}
	\begin{thm}\label{multiple_of_6}
		For any subsets $A$ of $\mathbb{F}_{2^n}\backslash\{0\}$, $r(A)\equiv0\pmod 6$.
	\end{thm}
	\begin{proof}
		Suppose $r(A)>0$. Then there exist $\{a,b,c\}$ in $A$ such that $a+b=c$. But then $a+c=b$ and $b+c=a$. Hence by abalian property $\{a,b,c\}$ contribute 6 pairs to $r(A)$. Suppose $\{d,e,f\}$ is another pair which is disjoint from $\{a,b,c\}$, then $\{d,e,f\}$ is also contribute 6 to $r(A)$. If $\{a,b,c\}\cap\{d,e,f\}\neq\emptyset$, then either both are equal or one element is common. If two elements are common, that is $a=d$ and $b=e$, then $c=a+b=d+e=f$. Suppose $a=d$, then again $\{a,e,f\}$ contribute 6 to $r(A)$ without overlapping with the combinations of $\{a,b,c\}$. Hence, the proof. 
	\end{proof}
	The proof of above result ensures that, if $\mathcal{R}=\{set(B)|B\in R(A,A,A)\}$, that is elements of $R$-set of $A$
	is considered as set, then $\mathcal{R}$ is a blocks of $A$ to forms partial Steiner triple system with $|\mathcal{R}|=r(A)/6$.

	\begin{prop}\label{sum_free}
		Let $B=\{a\in\mathbb{F}_{2^n}|Tr(a)=1\}$. Then $r(B)=0$.
	\end{prop}
	\begin{proof}
		Let $a,b\in B$, then $Tr(a+b)=Tr(a)+Tr(b)=1+1=0$, which implies $a+b\notin B$. Thus $r(B)=0$. 
	\end{proof}
	In the above proposition, note that $0\notin B$ and every subset of $B$ is sum-free set.  In other words, for each subset of size $m$ with $0\leq m\leq |B|=2^{n-1}$ there is a subset of $\mathbb{F}_{2^n}$ of zero  $r$-value.
	
	\begin{thm}\label{zero_added}
		For any subsets $A$ of $\mathbb{F}_{2^n}\backslash\{0\}$ with $|A|=l$, we have $r(A\cup\{0\})=r(A)+3l+1$.
	\end{thm}
	\begin{proof}
		We have $r(A\cup\{0\})=r(A)+r(A,\{0\})+2r(A,\{0\},A)+2r(A,\{0\},\{0\})+r(\{0\},A)+r(\{0\})$. Since for any $a\in A$, its additive inverse is $a$ itself, we have $r(A,\{0\})=|A|=l$. Now 0 is the identity element and $0\notin A$, we have $r(A,\{0\},A)=|A|=l$, $r(A,\{0\},\{0\})=0$ and $r(\{0\},A)=0$. Hence $r(A\cup\{0\})=r(A)+3l+1$.
	\end{proof}
	The above theorem help us to calculate the $r$-value of a subset $A$ of $\mathbb{F}_{2^n}$ with $0\in A$ by knowing the $r$-value of $A\backslash\{0\}$. The following result gives the $r$-values of size 4 subsets in $\mathbb{F}_{2^n}\backslash\{0\}$. 
	
	\begin{prop}\label{size 4}
		If $A\subseteq \mathbb{F}_{2^n}$ with $|A|=4$ and $0\notin A$, then $r(A)$ is either 0 or 6.
	\end{prop}
	\begin{proof}
		Suppose $A$ is a subset of $B$, where $B$ is as defined in Proposiion \ref{sum_free}. Then $r(A)=0$.  Suppose $r(A)>0$, then using Proposition \ref{3_subsset}, there exits a three elements $\{a,b,c\}$ of $A$ contribute 6 pairs to $r(A)$. 
		If there exists another set $\{d,e,f\}$ then as seen from proof of Theorem \ref{multiple_of_6}, there can be atmost one element common in $\{a,b,c\}$ and $\{d,e,f\}$. Which contradicts $|A|=4$. Hence, the proof.
	\end{proof}
	The following Theorem \ref{Upper_bound} gives the upper bound for the $r$-values of subsets of $\mathbb{F}_{2^n}\backslash\{0\}$.
	\begin{thm}\label{Upper_bound}
		For any subsets $A$ of $\mathbb{F}_{2^n}\backslash\{0\}$ with $|A|=k$, then $r(A)\leq \lfloor\frac{k(k-1)}{6}\rfloor6$. Moreover $r(A)=k(k-1)$ holds if and only if $A\cup\{0\}$ is a subgroup of $\mathbb{F}_{2^n}$.
	\end{thm}
	\begin{proof}
		Clearly we have $r(A)\leq k^2-k$. But by Theorem \ref{multiple_of_6}, $r(A)$ is a multiple of 6. Thus $r(A)\leq \lfloor\frac{k(k-1)}{6}\rfloor6$. Moreover, $A\cup\{0\}$ is a subgroup of $\mathbb{F}_{2^n}$ if and only in $r(A\cup\{0\})=(k+1)^2$. By Theorem \ref{zero_added} $r(A\cup\{0\})=(k+1)^2$ if and only if $r(A)=(k+1)^2-(3k+1)=k(k-1)$.
	\end{proof}
	
	\subsection{Spectrum of $r$-values in $\mathbb{F}_{2^3}$}\label{F_2_3}
	The above listed results are enough to generate a spectrum of $r$-values in  $\mathbb{F}_{2^3}$.
	\begin{figure}[h]
		\centering
		\includegraphics[width=0.3\linewidth]{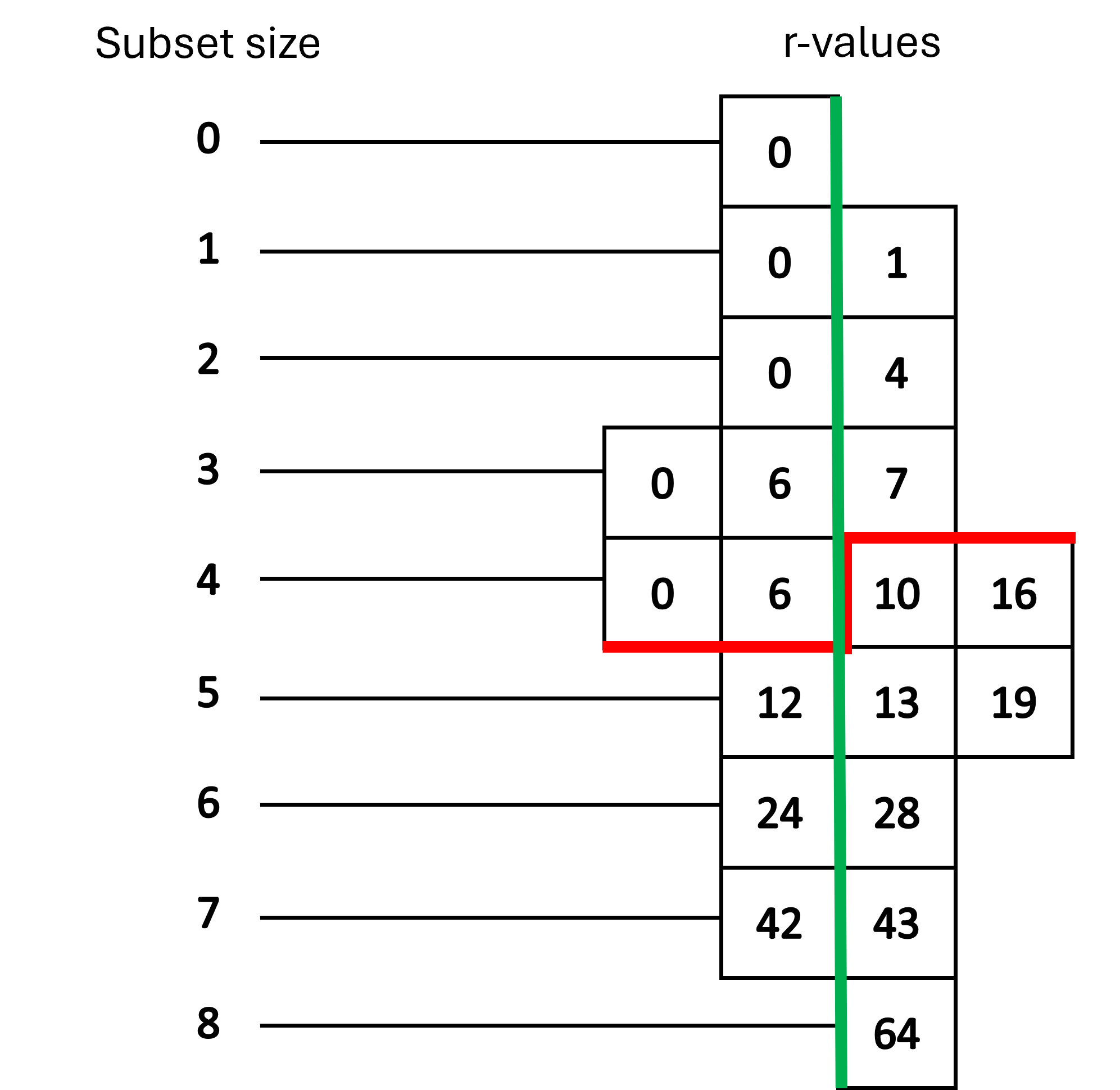}
		\caption{Spectrum of $r$-values in  $\mathbb{F}_{2^3}$}
		\label{fig:table-231}
	\end{figure}
	The Figure \ref{fig:table-231} gives the all possible $r$-values of each subset size. 
	
	The $r$-values of subsets of $\mathbb{F}_{2^3}\backslash\{0\}$ is tabulated in the left side of vertical green coloured line. The values in the right side of vertical green coloured line gives the $r$-values of subsets containing 0. Moreover the right side values are calculated using Theorem \ref{zero_added}. For example to calculate the $r$-values of subsets of size 4 containing 0 is obtained by adding $3\times 4-2=10$ to each $r$-values of the subsets of size 3 in $\mathbb{F}_{2^3}\backslash\{0\}$. The values listed below the red coloured line can be computed using Theorem \ref{compliment}. Thus by knowing $r$-values of subsets of $\mathbb{F}_{2^3}\backslash\{0\}$ of size upto 4, it is possible to generate spectrum of $r$-values for each size in $\mathbb{F}_{2^3}$.
	
	\section{Spectrum of $r$-values in $\mathbb{F}_{2^n}$}\label{Spec 2_n}
	In this section we generate spectrum of $r$-values in $\mathbb{F}_{2^n}$, where $n\geq 4$ using spectrum of $r$-values in $\mathbb{F}_{2^{n-1}}$. As seen in the subsection \ref{F_2_3}, to generate spectrum of $r$-values in $\mathbb{F}_{2^n}$ it is enough to calculate the $r$-values of subsets of $\mathbb{F}_{2^n}\backslash\{0\}$ upto size $2^{n-1}$.

	The following theorem gives the condensed formula for $r$-value of $A\cup B$ in $\mathbb{F}_{2^n}$.
	
	\begin{thm}\label{union}
		Suppose $A\subseteq\mathbb{F}_{2^{n-1}}\backslash\{0\}$ and $B\subseteq\mathbb{F}_{2^n}\backslash\mathbb{F}_{2^{n-1}}$, then $r(A\cup B)=r(A)+3r(A,B,B)$.
	\end{thm}
	\begin{proof}
		We have $r(A\cup B)=r(A)+r(A,B)+2r(A,B,A)+2r(A,B,B)+r(B,A)+r(B)$. Since $A\subseteq\mathbb{F}_{2^{n-1}}\backslash\{0\}$ and $B\subseteq\mathbb{F}_{2^n}\backslash\mathbb{F}_{2^{n-1}}$, we observed that sum of two elements in $A$ cannot lies in $B$. Which gives $r(A,B)=0$. In vector representation, the first coordinate of each element of $A$ is $0$, and the first coordinate of each element of $B$ is $1$. Thus, the sum of two elements in $B$ cannot belongs to $B$, and if one element is from $A$ and another element is from $B$, then the sum cannot belongs to $A$. Therefor $r(B)=0$ and $r(A,B,A)=0$.
		
		Since each element is its additive inverse in $\mathbb{F}_{2^n}$ we have,	
		\begin{align*}
			(\alpha,\beta,\gamma)\in R(A,B,B)&\Leftrightarrow(\alpha,\beta,\gamma),(\alpha,\gamma,\beta)\in R(A,B,B)\\
			&\Leftrightarrow(\beta,\alpha,\gamma),(\gamma,\alpha,\beta)\in R(B,A,B)\\
			&\Leftrightarrow(\beta,\gamma,\alpha),(\gamma,\beta,\alpha)\in R(B,B,A)
		\end{align*}
		
		Thus $|R(A,B,B)|=|R(B,A,B)|=|R(B,B,A)|$ and $r(A,B,B)=r(B,A,B)=r(B,A)$. Hence $r(A\cup B)=r(A)+3r(A,B,B)$.
	\end{proof}
	\begin{cor}\label{coro}
		Every subset of $\mathbb{F}_{2^n}\backslash\mathbb{F}_{2^{n-1}}$ has $r$-value zero.
	\end{cor}
	The proof of the above Corollary directly follows from the Theorem \ref{union}. The following result relate the $r$-value of a set and its compliment in $\mathbb{F}_{2^n}\backslash\{0\}$.
	\begin{prop}\label{compliment without zero}
		Let $S\subseteq\mathbb{F}_{2^n}\backslash\{0\}$ with size $k$. Then the compliment of $S$ in $\mathbb{F}_{2^n}\backslash\{0\}$ denoted by $\overline{S}^*$ is a subset of size $2^n-1-k$ with $r$-value  given by $r(\overline{S}^*)=2^{2n}-3(2^n-k-1)(k+1)-(3k+1)-r(S)$.
	\end{prop}
	
	\begin{proof}
		Let $S\subseteq\mathbb{F}_{2^n}\backslash\{0\}$ with size $k$ and $\overline{S}$ is the compliment of $S$ in $\mathbb{F}_{2^n}$. Then by Theorem \ref{compliment}, $r(\overline{S})=2^{2n}-3\cdot2^nk+3k^2-r(S)$ with $|\overline{S}|=2^n-k$. Now let $\overline{S}^*=\overline{S}\backslash\{0\}$, then $\overline{S}^*$ is compliment of $S$ in $\mathbb{F}_{2^n}\backslash\{0\}$ with size $2^n-k-1$ and by Proposition \ref{bound} $r(\overline{S}^*)=r(\overline{S})-(3(2^n-1-k)+1)=2^{2n}-3(2^n-k-1)(k+1)-(3k+1)-r(S)$.
	\end{proof}
	Using the above result it is enough to calculate the $r$-values of subsets of $\mathbb{F}_{2^n}\backslash\{0\}$ upto size $2^{n-1}$.
	\par Let $S$ be a subset of $\mathbb{F}_{2^n}\backslash\{0\}$ with size $m$, where $5\leq m\leq 2^{n-1}$. Suppose $S=A\cup B$ where $A\subseteq\mathbb{F}_{2^{n-1}}\backslash\{0\}$ with $|A|=k$ and $B\subseteq\mathbb{F}_{2^n}\backslash\mathbb{F}_{2^{n-1}}$ with $|B|=l$, where $0\leq k\leq 2^{n-1}-1 $ and $0\leq l\leq 2^{n-1}$.

	\par We continue by considering different cases on $k$ and $l$. The case $k=0$ is follows from the Proposition \ref{sum_free} and Corollary \ref{coro} by giving zero as an $r$-value in each subset size from 1 to $2^{n-1}$ in $\mathbb{F}_{2^n}$. The following result gives the $r$-value when $l=0$ and $l=1$.  
	\begin{prop}\label{l=0,1}
		Let $A\subseteq\mathbb{F}_{2^{n-1}}\backslash\{0\}$ with size $k$. Then $r(A)$ is a $r$-value of two subsets of size $k$ and $k+1$ in $\mathbb{F}_{2^{n}}\backslash\{0\}$.
	\end{prop}
	\begin{proof}
		Let $A=\{a_1,a_2,\dots,a_k\}$ be a subset of $\mathbb{F}_{2^{n-1}}\backslash\{0\}$. Then $r(A)$ in $\mathbb{F}_{2^n}$ is equal to $r(A)$ in $\mathbb{F}_{2^{n-1}}$. In other words, each $r$-values occurring in $\mathbb{F}_{2^{n-1}}\backslash\{0\}$ is a $r$-value of subsets in $\mathbb{F}_{2^n}\backslash\{0\}$ with same subset size.
		
		Let $b\in\mathbb{F}_{2^n}\backslash\mathbb{F}_{2^{n-1}}$ be fixed. Then we have $r(\{a_1,a_2,\dots,a_{k},b\})=r(\{a_1,a_2,\dots,a_{k}\})+3r(\{a_1,a_2,\dots,a_{k}\},\{b\},\{b\})$. But then $a_i+b\neq b$ for all $1\leq i\leq k$, we get $r(\{a_1,a_2,\dots,a_{k},b\})=r(\{a_1,a_2,\dots,a_{k}\})$. Here note that $\{a_1,a_2,\dots,a_{k}\}$ is a subset of size $k$ in $\mathbb{F}_{2^{n-1}}\backslash\{0\}$, giving an $r$-value of size $k+1$ subset in $\mathbb{F}_{2^n}\backslash\{0\}$.
	\end{proof}
	\par Suppose $l=2$, then we have the following result. 
	\begin{prop}\label{l=2}
		Suppose $S=\{a_1,a_2,\dots,a_k,b_1,b_2\}$, where $a_1,a_2,\dots,a_k\in \mathbb{F}_{2^{n-1}}\backslash\{0\}$, $k\geq2$ and $b_1,b_2\in \mathbb{F}_{2^n}\backslash\mathbb{F}_{2^{n-1}}$.  Then $r(\{a_1,a_2,\dots,a_k,b_1,b_2\})=r(\{a_1,a_2,\dots,a_k\})+\begin{cases}
			6 &\text{if }b_1+b_2\in \{a_1,a_2,\dots,a_k\}, \\
			0 & otherwie.
		\end{cases}$
	\end{prop}
	\begin{proof}
		We have $k+2\leq2^{n-1}$ then $k\leq 2^{n-1}-2$, which implies there is an element in $\mathbb{F}_{2^{n-1}}\backslash\{0\}$ say $a$ such that $a\notin\{a_1,a_2,\dots,a_k\}$. But then there exists two elements $b_1,b_2\in \mathbb{F}_{2^n}\backslash\mathbb{F}_{2^{n-1}}$ such that $b_1+b_2=a$. In this case we have $r(\{a_1,a_2,\dots,a_k,b_1,b_2\})=r(\{a_1,a_2,\dots,a_{m-2}\})$. Suppose $a=a_i$ for some $\leq i\leq k$, then $(b_1,b_2,a_i),(b_2,b_1,a_i)\in R(\{b_1,b_2\},\{b_1,b_2\},\{a_1,a_2,\dots,a_k\})$. Hence by Theorem \ref{union} we have, $r(\{a_1,a_2,\dots,a_k,b_1,b_2\})=r(\{a_1,a_2,\dots,a_k\})+6$.
	\end{proof}
	From the above result, each $r$-value $r$ of subsets of $\mathbb{F}_{2^{n-1}}\backslash\{0\}$ of size $k$ where $1\leq k\leq 2^{n-1}-2$, $r$ and $r+6$ are $r$-values of size $k+2$ in $\mathbb{F}_{2^n}\backslash\{0\}$.	
	
	\begin{thm}\label{ABB}
		There exits two subsets $A$ of $\mathbb{F}_{2^{n-1}}\backslash\{0\}$ and $B$ of $\mathbb{F}_{2^n}\backslash\mathbb{F}_{2^{n-1}}$ with same size $2^{n-2}$ such that $r(A,B,B)=0$. Moreover, there is no such subsets of size greater than $2^{n-2}$.
	\end{thm}
	\begin{proof}
		Let $A=\{[\alpha_1,\alpha_2,\dots,\alpha_n]\in\mathbb{F}_{2^n}~|~\alpha_1=0,\alpha_n=1\}$ and $B=\{[\alpha_1,\alpha_2,\dots,\alpha_n]\in\mathbb{F}_{2^n}~|~\alpha_1=1,\alpha_n=0\}$. Clearly $|A|=|B|=2^{n-2}$. Let $C=\{[\alpha_1,\alpha_2,\dots,\alpha_n]\in\mathbb{F}_{2^n}~|~\alpha_1=\alpha_n=1\}$. Note that for any $a\in A$ and $b\in B$ we have $a+b\in C$. Therefore $r(A,B,B)=0$. 
		
		Suppose $A\subseteq\mathbb{F}_{2^{n-1}}\backslash\{0\}$ with $|A|=2^{n-2}+1$. Then there is an element $a$ in $A$ of the form $a=[0,\alpha_2,\dots,\alpha_{n-1},0]$. But then for any $[1,\beta_2,\dots,\beta_{n-1},0]\in B$,  we have  $[0,\alpha_2,\dots,\alpha_{n-1},0]+[1,\beta_2,\dots,\beta_{n-1},0]=[1,\alpha_2+\beta_2,\dots,\alpha_{n-1}+\beta_{n-1},0]\in B$. Which gives $r(A,B,B)>0$.
		
		Similarly suppose $B\subseteq\mathbb{F}_{2^n}\backslash\mathbb{F}_{2^{n-1}}$ with $|B|=2^{n-2}+1$. Then there is an element $b$ in $B$ of the form $b=[1,\beta_2,\dots,\beta_{n-1},1]$. But then for any $[0,\alpha_2,\dots,\alpha_{n-1},1]\in A$,  we have  $[0,\alpha_2,\dots,\alpha_{n-1},1]+[1,\beta_2,\dots,\beta_{n-1},1]=[1,\alpha_2+\beta_2,\dots,\alpha_{n-1}+\beta_{n-1},0]\in B$ giving $r(A,B,B)>0$.
	\end{proof}
	Suppose $k=1$, then we have the following result. 
	\begin{prop}\label{k=1}
		Suppose $S=\{a_1,b_1,\dots,b_{m-1}\}$, where $a_1\in \mathbb{F}_{2^{n-1}}\backslash\{0\}$ and $b_1,\dots,b_{m-1}\in \mathbb{F}_{2^n}\backslash\mathbb{F}_{2^{n-1}}$.  Then $r(\{a_1,b_1,\dots,b_{m-1}\})$ attain each values in the set $\left\{6i~|0\leq i\leq \left\lfloor\frac{m-1}{2}\right\rfloor\right\}$ whenever $m-1\leq2^{n-2}$ and whenever $m-1>2^{n-2}$, $r(\{a_1,b_1,\dots,b_{m-1}\})$ attain each values in the set $\left\{6i~|(m-1)-2^{n-2}\leq i\leq \left\lfloor\frac{m-1}{2}\right\rfloor\right\}$.
	\end{prop}
	\begin{proof}
		Suppose $m-1\leq2^{n-2}$. If $a_1\in A$ and $\{b_1,\dots,b_{m-1}\}\subseteq B$, where $A$ and $B$ as defined in Theorem \ref{ABB}, then $r(\{a_1\},\{b_1,\dots,b_{m-1}\},\{b_1,\dots,b_{m-1}\})=0$ giving $r(\{a_1,b_1,\dots,b_{m-1}\})=0$. We have if $(a_1,b_i,b_j)\in R(\{a_1\},\{b_1,\dots,b_{m-1}\},\{b_1,\dots,b_{m-1}\})$ for some $1\leq i,j\leq m-1$ then $(a_1,b_j,b_i)\in R(\{a_1\},\{b_1,\dots,b_{m-1}\},\{b_1,\dots,b_{m-1}\})$. Also for each $a_i\in\mathbb{F}_{2^{n-1}}\backslash\{0\}$ there are $2^{n-2}$ pairs $\{b_t,b_s\}$, where $b_t,b_s\in\mathbb{F}_{2^n}\backslash\mathbb{F}_{2^{n-1}}$ such that $a_i=b_t+b_s$. Using this, among the $m-1$  elements $\{b_1,\dots,b_{m-1}\}$, atmost $\lfloor\frac{m-1}{2}\rfloor$ possible to give sum as $a_1$. Hence by Theorem \ref{union} $r(\{a_1,b_1,\dots,b_{m-1}\})$ attain each values in the set $\left\{6i~|0\leq i\leq \left\lfloor\frac{m-1}{2}\right\rfloor\right\}$. 
		
		Suppose $m-1>2^{n-2}$. For distinct $i$ and $j$ with $1\leq i,j\leq m-1$, we know that $a_1+b_i$ and $a_1+b_j$ distinct and they all lies in $\mathbb{F}_{2^n}\backslash\mathbb{F}_{2^{n-1}}$. Since  $m-1>2^{n-2}$ and $|\mathbb{F}_{2^n}\backslash\mathbb{F}_{2^{n-1}}|=2^{n-1}$ there are atleast $2((m-1)-2^{n-2})$ number of $i$'s between 1 and $m-1$ such that $a_1+b_i\in\{b_1,\dots,b_{m-1}\}$. Now as seen in above case, among the $m-1$  elements $\{b_1,\dots,b_{m-1}\}$, atmost $\lfloor\frac{m-1}{2}\rfloor$ possible to give sum as $a_1$. Hence in this case $r(\{a_1,b_1,\dots,b_{m-1}\})$ attain each values in the set $\left\{6i~|(m-1)-2^{n-2}\leq i\leq \left\lfloor\frac{m-1}{2}\right\rfloor\right\}$.
		
	\end{proof}
	The Proposition \ref{3-18},\ref{3-12},\ref{3-6} and \ref{3-0} gives $r$ value in the case $l=3$.
	\begin{prop}\label{3-18}
		Let $A\subseteq\mathbb{F}_{2^{n-1}}\backslash\{0\}$ with size $k\geq3$ and positive $r$-value. Then $r(A)+18$ is a $r$-value of size $k+3$ subset in $\mathbb{F}_{2^{n}}\backslash\{0\}$.
	\end{prop}
	\begin{proof}
		Let $A=\{a_1,a_2,\dots,a_k\}\subseteq\mathbb{F}_{2^{n-1}}\backslash\{0\}$ with size $k\geq3$ and positive $r$-value. Since $r(A)>0$, there exists three elements say $a_1,a_2,a_3$ such that $a_1+a_2=a_3$ in $\mathbb{F}_{2^n}$. 
		
		Let $b_1=a_1+[1,0,0,\dots,0]$, $b_2=a_2+[1,0,0,\dots,0]$ and $b_3=a_3+[1,0,0,\dots,0]$. Note that $b_1,b_2,b_3\in\mathbb{F}_{2^n}\backslash\mathbb{F}_{2^{n-1}}$ with $b_1+b_2=a_3$, $b_1+b_3=a_2$ and $b_2+b_3=a_1$ in $\mathbb{F}_{2^n}$. Thus if $B=\{b_1,b_2,b_3\}$, then $r(B,B,A)=6$. Hence by Theorem \ref{union} $r(A\cup B)=r(A)+18$.
	\end{proof}
	
	\begin{prop}\label{3-12}
		Let $A\subseteq\mathbb{F}_{2^{n-1}}\backslash\{0\}$ with size $k\geq2$ and $r(A)\neq k(k-1)$. Then $r(A)+12$ is a $r$-value of size $k+3$ subset in $\mathbb{F}_{2^{n}}\backslash\{0\}$.
	\end{prop}
	\begin{proof}
		
		Given $k\geq2$ and $r(A)\neq k(k-1)$, there exist two elements say $a_1, a_2$ such that $a_1+a_2\notin A$. Let $b_1=a_1+[1,0,0,\dots,0]$, $b_2=a_2+[1,0,0,\dots,0]$ and $b_3=a_1+a_2+[1,0,0,\dots,0]$. Note that $b_1+b_3=a_2\in A$ and $b_2+b_3=a_1\in A$, but $b_1+b_2=a_1+a_2\notin A$. Thus $r(B,B,A)=4$ and by Theorem \ref{union} $r(A\cup \{b_1,b_2,b_3\})=r(A)+12$.
	\end{proof}
	
	\begin{prop}\label{3-6}
		Let $A\subseteq\mathbb{F}_{2^{n-1}}\backslash\{0\}$ with size $0<k<2^{n-1}-3$ and $\overline{A}^*$ is the compliment of $A$ in $\mathbb{F}_{2^{n-1}}\backslash\{0\}$ with $r(\overline{A}^*)\neq|\overline{A}^*|(|\overline{A}^*|-1)$. Then $r(A)+6$ is a $r$-value of size $k+3$ subset in $\mathbb{F}_{2^{n}}\backslash\{0\}$.
	\end{prop}
	\begin{proof}
		
		Given $r(\overline{A}^*)\neq|\overline{A}^*|(|\overline{A}^*|-1)$, there exist three elements say $\overline{a}_1, \overline{a}_2$ in $\overline{A}^*$ and $a_3\in A$ such that $\overline{a}_1+\overline{a}_2=a_3$. Let $b_1=\overline{a}_1+[1,0,0,\dots,0]$, $b_2=\overline{a}_2+[1,0,0,\dots,0]$ and $b_3=a_3+[1,0,0,\dots,0]$. Note that $b_1+b_2=a_3\in A$, $b_1+b_3=\overline{a}_2\notin A$ and $b_2+b_3=\overline{a}_1\notin A$. Thus $r(B,B,A)=2$ and by Theorem \ref{union} $r(A\cup \{b_1,b_2,b_3\})=r(A)+6$.
	\end{proof}
	
	\begin{prop}\label{3-0}
		Let $A\subseteq\mathbb{F}_{2^{n-1}}\backslash\{0\}$ with size $k<2^{n-1}-4$ and $\overline{A}^*$ is the compliment of $A$ in $\mathbb{F}_{2^{n-1}}\backslash\{0\}$ with $r(\overline{A}^*)>0$. Then $r(A)$ is a $r$-value of size $k+3$ subset in $\mathbb{F}_{2^{n}}\backslash\{0\}$.
	\end{prop}
	\begin{proof}
		Given $r(\overline{A}^*)>0$, there exist three elements say $\overline{a}_1, \overline{a}_2,\overline{a}_3$ in $\overline{A}^*$ such that $\overline{a}_1+\overline{a}_2=\overline{a}_3$. Let $b_1=\overline{a}_1+[1,0,0,\dots,0]$, $b_2=\overline{a}_2+[1,0,0,\dots,0]$ and $b_3=\overline{a}_3+[1,0,0,\dots,0]$. Note that $b_1+b_2=\overline{a}_3\in A$, $b_1+b_3=\overline{a}_2$ and $b_2+b_3=\overline{a}_1$. Thus $r(B,B,A)=0$ and by Theorem \ref{union} $r(A\cup \{b_1,b_2,b_3\})=r(A)$.
	\end{proof}
	When $l=4$, $r$-values of subset size $k+4$ is obtained using the results \ref{4-36},\ref{4-24},\ref{4-12} and \ref{4-0}.
	\begin{prop}\label{4-36}
		Let $A\subseteq\mathbb{F}_{2^{n-1}}\backslash\{0\}$ with size $k$ and positive $r$-value.  Then $r(A)+36$ is a $r$-value of size $k+4$ subset in $\mathbb{F}_{2^{n}}\backslash\{0\}$.
	\end{prop}
	\begin{proof}
		Given $r(A)>0$, there exist three elements say $a_1,a_2,a_3\in A$ such that $a_1+a_2=a_3$. Let $b\in\mathbb{F}_{2^n}\backslash\mathbb{F}_{2^{n-1}}$ be fixed and $b_1=b+a_1$, $b_2=b+a_2$ and $b_3=b+a_3$. Note that all three $b_1,b_2,b_3$ are distinct. If $B=\{b,b_1,b_2,b_3\}$, then $r(B,B,A)=12$. Hence by Theorem \ref{union} $r(A\cup B)=r(A)+36$.
	\end{proof}
	
	\begin{prop}\label{4-24}
		Let $A\subseteq\mathbb{F}_{2^{n-1}}\backslash\{0\}$ with size $k\geq2$ and $r(A)\neq k(k-1)$. Then $r(A)+24$ is a $r$-value of size $k+4$ subset in $\mathbb{F}_{2^{n}}\backslash\{0\}$.
	\end{prop}
	\begin{proof}
		Given $k\geq2$ and $r(A)\neq k(k-1)$, there exist two elements say $a_1,a_2\in A$ such that $a_1+a_2\notin A$. Let $b\in\mathbb{F}_{2^n}\backslash\mathbb{F}_{2^{n-1}}$ be fixed and $b_1=b+a_1$, $b_2=b+a_2$ and $b_3=b+a_1+a_2$. Note that all three $b_1,b_2,b_3$ are distinct. If $B=\{b,b_1,b_2,b_3\}$, then $r(B,B,A)=8$. Hence by Theorem \ref{union} $r(A\cup B)=r(A)+24$.
	\end{proof}
	
	\begin{prop}\label{4-12}
		Let $A\subseteq\mathbb{F}_{2^{n-1}}\backslash\{0\}$ with size $0<k<2^{n-1}-3$ and $\overline{A}^*$ is the compliment of $A$ in $\mathbb{F}_{2^{n-1}}\backslash\{0\}$ with $r(\overline{A}^*)\neq|\overline{A}^*|(|\overline{A}^*|-1)$. Then $r(A)+12$ is a $r$-value of size $k+4$ subset in $\mathbb{F}_{2^{n}}\backslash\{0\}$.
	\end{prop}
	\begin{proof}
		
		Given $r(\overline{A}^*)\neq|\overline{A}^*|(|\overline{A}^*|-1)$, there exist three elements say $\overline{a}_1, \overline{a}_2 \in \overline{A}^*$ and $a_3\in A$ such that $\overline{a}_1+\overline{a}_2=a_3$. Let $b\in\mathbb{F}_{2^n}\backslash\mathbb{F}_{2^{n-1}}$ be fixed and $b_1=b+\overline{a}_1$, $b_2=b+\overline{a}_2$ and $b_3=b+a_3$. Note that all three $b_1,b_2,b_3$ are distinct. If $B=\{b,b_1,b_2,b_3\}$, then $r(B,B,A)=4$. Hence by Theorem \ref{union} $r(A\cup B)=r(A)+12$.
	\end{proof}
	\begin{prop}\label{4-0}
		Let $A\subseteq\mathbb{F}_{2^{n-1}}\backslash\{0\}$ with size $k<2^{n-1}-4$ and $\overline{A}^*$ is the compliment of $A$ in $\mathbb{F}_{2^{n-1}}\backslash\{0\}$ with $r(\overline{A}^*)>0$. Then $r(A)$ is a $r$-value of size $k+4$ subset in $\mathbb{F}_{2^{n}}\backslash\{0\}$.
	\end{prop}
	\begin{proof}
		Given $r(\overline{A}^*)>0$, there exist three elements say $\overline{a}_1,\overline{a}_2,\overline{a}_3\in \overline{A}^*$ such that $\overline{a}_1+\overline{a}_2=\overline{a}_3$. Let $b\in\mathbb{F}_{2^n}\backslash\mathbb{F}_{2^{n-1}}$ be fixed and $b_1=b+\overline{a}_1$, $b_2=b+\overline{a}_2$ and $b_3=b+\overline{a}_3$. Note that all three $b_1,b_2,b_3$ are distinct. If $B=\{b,b_1,b_2,b_3\}$, then $r(B,B,A)=0$. Hence by Theorem \ref{union} $r(A\cup B)=r(A)$.
	\end{proof}
	The following result cover the case $k=2$ and $l=5$ to give $r$-values of subsets of size 7 in $\mathbb{F}_{2^n}\backslash\{0\}$.
	\begin{prop}\label{2+5}
		Suppose $S=\{a_1,a_2,b_1,b_2,b_3,b_4,b_5\}$, where $a_1,a_2\in \mathbb{F}_{2^{n-1}}\backslash\{0\}$ and $b_1,b_2,b_3,b_4,b_5\in \mathbb{F}_{2^n}\backslash\mathbb{F}_{2^{n-1}}$.  Then for $n\geq 5$, $r(S)$ will attain the values 0,6,12,18 and 24. For $n=4$, $r(S)$ will attain the values 0,12,18 and 24.
	\end{prop}
	
	\begin{proof}
		Let $A=\{a_1,a_2\}$ and $B=\{b_1,b_2,b_3,b_4,b_5\}$.
		Using the following Cayley Table \ref{Type 1}, $r(B,B,A)$ depends on $r(\{x_1,x_2,x_3,x_4\},\{x_1,x_2,x_3,x_4\},A)$.	\\
		\begin{table}[h]
			\caption{Type 1}
			\label{Type 1}
			\begin{center}
				\begin{tabular}{c||c|c|c|c|c}
					
					&$b_1$ &$b_2$ &$b_3$ &$b_4$ &$b_5$ \\
					\hline
					\hline
					$b_1$ &0 &$x_1$ &$x_2$ &$x_3$ &$x_4$ \\
					\hline
					$b_2$&$x_1$ &0 &$x_1+x_2$ &$x_1+x_3$ &$x_1+x_4$ \\
					\hline
					$b_3$&$x_2$ &$x_1+x_2$ &0 &$x_2+x_3$ &$x_2+x_4$ \\
					\hline
					$b_4$&$x_3$ &$x_1+x_3$ &$x_2+x_3$ &0 &$x_3+x_4$ \\
					\hline
					$b_5$&$x_4$ &$x_1+x_4$ &$x_2+x_4$ & $x_3+x_4$&0 \\
				\end{tabular}
			\end{center}
		\end{table}
		Since  $x_1,x_2,x_3,x_4\in\mathbb{F}_{2^{n-1}}\backslash\{0\}$, we have $r(\{x_1,x_2,x_3,x_4\})$ is 0 or 6. If $r(\{x_1,x_2,x_3,x_4\})=0$, then $x_1+x_2,x_1+x_3,x_1+x_4\notin\{x_1,x_2,x_3,x_4\}$. Say $x_1+x_2=x_5,x_1+x_3=x_6$ and $x_1+x_4=x_7$.
		
		Suppose $r$-value of $\{x_5,x_6,x_7\}$ is 0, then $x_5+x_6,x_5+x_7,x_6+x_7$ not in $\{x_5,x_6,x_7\}$ and $\{x_1,x_2,x_3,x_4\}$. Say $x_5+x_6=x_2+x_3=x_8$, $x_5+x_7=x_2+x_4=x_9$ and $x_6+x_7=x_3+x_4=x_{10}$. This is possible only if $n\geq 5$ (See Cayley Table \ref{Type 2}).
		\begin{table}[h]
			\caption{Type 2}
			\label{Type 2}
			\begin{center}
				\begin{tabular}{c||c|c|c|c|c}
					
					&$b_1$ &$b_2$ &$b_3$ &$b_4$ &$b_5$ \\
					\hline
					\hline
					$b_1$ &0 &$x_1$ &$x_2$ &$x_3$ &$x_4$ \\
					\hline
					$b_2$&$x_1$ &0 &$x_5$ &$x_6$ &$x_7$ \\
					\hline
					$b_3$&$x_2$ &$x_5$ &0 &$x_8$ &$x_9$ \\
					\hline
					$b_4$&$x_3$ &$x_6$ &$x_8$ &0 &$x_{10}$ \\
					\hline
					$b_5$&$x_4$ &$x_7$ &$x_9$ & $x_{10}$&0 \\
				\end{tabular}
			\end{center}
		\end{table}
		If $\{a_1,a_2\}\not\subset\{x_1,x_2,\dots,x_{10}\}$, then $r(S)=0$. If one element of $\{a_1,a_2\}$ in $\{x_1,x_2,\dots,x_{10}\}$ and other one not in $\{x_1,x_2,\dots,x_{10}\}$, then $r(S)=6$. If $\{a_1,a_2\}\subset\{x_1,x_2,\dots,x_{10}\}$, then $r(S)=12$. 
		
		Suppose $r$-value of $\{x_5,x_6,x_7\}$ is 6, then Cayley Table \ref{Type 2} becomes Table \ref{Type 3}, 
		\begin{table}[h]
			\caption{Type 3}
			\label{Type 3}
			\begin{center}
				\begin{tabular}{c||c|c|c|c|c}
					
					&$b_1$ 	&$b_2$ &$b_3$ &$b_4$ &$b_5$ \\
					\hline
					\hline
					$b_1$	&0		&$x_1$ 	&$x_2$ 	&$x_3$ 		&$x_4$ \\
					\hline
					$b_2$	&$x_1$ 	&0 		&$x_5$ 	&$x_6$ 		&$x_7$ \\
					\hline
					$b_3$	&$x_2$ 	&$x_5$ 	&0 		&$x_7$ 		&$x_6$ \\
					\hline	
					$b_4$	&$x_3$ 	&$x_6$ 	&$x_7$ 	&0 			&$x_{5}$ \\
					\hline
					$b_5$	&$x_4$ 	&$x_7$ 	&$x_6$ 	&$x_{5}$	&0 \\
				\end{tabular}
			\end{center}
		\end{table}
		If $\{a_1,a_2\}\not\subset\{x_1,x_2,x_3,x_4,x_5,x_6,x_7\}$, then $r(S)=0$. If one element of $\{a_1,a_2\}$ in $\{x_1,x_2,x_3,x_4\}$ and other one not in $\{x_1,x_2,x_3,x_4,x_5,x_6,x_7\}$, then $r(S)=6$. If one element of $\{a_1,a_2\}$ in $\{x_5,x_6,x_7\}$ and other one not in $\{x_1,x_2,x_3,x_4,x_5,x_6,x_7\}$, then $r(S)=12$. If $\{a_1,a_2\}\subset\{x_1,x_2,x_3,x_4\}$, then $r(S)=12$. If $\{a_1,a_2\}\subset\{x_5,x_6,x_7\}$, then $r(S)=24$. If one element of $\{a_1,a_2\}$ in $\{x_1,x_2,x_3,x_4\}$ and other one in $\{x_5,x_6,x_7\}$, then $r(S)=18$.
		
		If $r(\{x_1,x_2,x_3,x_4\})=6$, then Cayley table will be similar to following Table \ref{Type 4}. 
		\begin{table}[h]
			\caption{Type 4}
			\label{Type 4}
			\begin{center}
				\begin{tabular}{c||c|c|c|c|c}
					
					&$b_1$ 	&$b_2$ &$b_3$ &$b_4$ &$b_5$ \\
					\hline
					\hline
					$b_1$	&0		&$x_1$ 	&$x_2$ 	&$x_3$ 	&$x_4$ \\
					\hline
					$b_2$	&$x_1$ 	&0 		&$x_3$ 	&$x_2$ 	&$x_5$ \\
					\hline
					$b_3$	&$x_2$ 	&$x_3$ 	&0 		&$x_1$ 	&$x_6$ \\
					\hline	
					$b_4$	&$x_3$ 	&$x_2$ 	&$x_1$ 	&0 		&$x_7$ \\
					\hline
					$b_5$	&$x_4$ 	&$x_5$ 	&$x_6$ 	&$x_7$	&0 \\
				\end{tabular}
			\end{center}
		\end{table}
		\\	This is similar  to the above case; hence $r(S)$ will be one of the values in the set $\{0,6,12,18,24\}$.
		
	\end{proof}
	
	The following two results holds the case $l=k$ and $l=k+1$.
	
	\begin{prop}\label{2k}
		Let $A\subseteq\mathbb{F}_{2^{n-1}}\backslash\{0\}$ with size $k\geq3$. Then $4r(A)$ is a $r$-value of size $2k$ subset in $\mathbb{F}_{2^{n}}\backslash\{0\}$.
	\end{prop}
	\begin{proof}
		Let $A=\{a_1,a_2,\dots,a_k\}\subseteq\mathbb{F}_{2^{n-1}}\backslash\{0\}$ with size $k\geq3$. Let $B=\{b_i=a_i+[1,0,0,\dots,0]|~1\leq i\leq k\}$. Now note that $(a_i,a_j,a_t)\in R(A,A,A)$ if and only if $(b_i,b_j,a_t)\in R(B,B,A)$. Thus $|R(A,A,A)|=|R(B,B,A)|$ and by Theorem \ref{union} $r(A\cup B)=4r(A)$.
	\end{proof}
	
	\begin{prop}\label{2k+1}
		Let $A\subseteq\mathbb{F}_{2^{n-1}}\backslash\{0\}$ with size $k$. Then $4r(A)+6k$ is a $r$-value of a subset of size $2k+1$ in $\mathbb{F}_{2^{n}}\backslash\{0\}$.
	\end{prop}
	\begin{proof}
		Let $A=\{a_1,a_2,\dots,a_k\}$ be a subset of $\mathbb{F}_{2^{n-1}}\backslash\{0\}$ with size $k$. Let $b\in\mathbb{F}_{2^n}\backslash\mathbb{F}_{2^{n-1}}$ be fixed.
		For each $a_i\in A$, there exist $b_i\in\mathbb{F}_{2^n}\backslash\mathbb{F}_{2^{n-1}}$ such that $b+b_i=a_i$. Say $B=\{b,b_1,b_2,\dots,b_k\}$. Now for all $1\leq i\leq k$, $b+b_i=b_i+b=a_i\in A$, giving $2k$ to $r(B,B,A)$. For $1\leq i,j\leq k$ with $i\neq j$, we have $b_i+b_j=a_i+a_j$. Thus, $b_i+b_j\in A$ if and only if $a_i+a_j\in A$. Hence $r(B,B,A)=r(A)+2k$ and by Theorem \ref{union} $r(A\cup B)=4r(A)+6k$.
	\end{proof}
	
	For each subsets size, all possible $r$-values are termed as a spectrum of $r$-values. The following subsections \ref{F_2_4} and \ref{F_2_5} gives the spectrum in $\mathbb{F}_{2^4}$ and $\mathbb{F}_{2^5}$ respectively.
	\subsection{Spectrum of $r$-values in $\mathbb{F}_{2^4}$}\label{F_2_4}
	For each subsets of $\mathbb{F}_{2^4}$, its $r$-value can be obtained from the results listed above. Let $S$ be a subset of $\mathbb{F}_{2^4}\backslash\{0\}$ of cardinality $m$. Suppose $S=A\cup B$, where $A\subset\mathbb{F}_{2^3}\backslash\{0\}$ with cardinality $k$ and  $B\subset\mathbb{F}_{2^4}\backslash\mathbb{F}_{2^3}$ with cardinality $l$ such that $m=k+l$. For $m=1,2,3,4$ we use  Propositions \ref{size 1},\ref{size 2},\ref{3_subsset},\ref{size 4} resp. For subset size 5,6,7 see Table \ref{table2-3-1}
	\begin{table}[h]\label{table2-3-1}	
		\caption{Spectrum of $r$-values in $\mathbb{F}_{2^4}$ of size $m=5,6,7$.}		
		\begin{center}
			\begin{tabular}{|c|c|c|c|c|}
				\hline
				$m$	&$k$ &$l$ &$r(S)$ & Spectrum \\
				\hline
				\multirow{6}{1cm}{\centering 5}	&0 &5 &0 &\multirow{6}{6cm}{\centering $\{0,6,12,13,19\}$}\\
				\cline{2-4}
				&1 &4 &0,6,12 & \\
				\cline{2-4}
				&2 &3 &0,6,12 &\\
				\cline{2-4}
				&3 &2 &0,6,12 & \\
				\cline{2-4}
				&4 &1 &0,6 & \\
				\cline{2-4}
				&5 &0 &12 & \\
				\hline
				\multirow{7}{1cm}{\centering 6}	&0 &6 &0 &\multirow{7}{6cm}{\centering $\{0,12,18,24,42,19,25,31,43\}$}\\
				\cline{2-4}
				&1 &5 &6,12 & \\
				\cline{2-4}
				&2 &4 &0,12,24 & \\
				\cline{2-4}
				&3 &3 &6,12,24 & \\
				\cline{2-4}
				&4 &2 &0,6,12& \\
				\cline{2-4}
				&5 &1 &12 & \\
				\cline{2-4}
				&6 &0 &24 & \\
				\hline
				\multirow{8}{1cm}{\centering 7}	&0 &7 &0 &\multirow{8}{6cm}{\centering $\{0,18,24,30,42,22,34,40,46,64\}$} \\
				\cline{2-4}
				&1 &6 &12,18 & \\
				\cline{2-4}
				&2 &5 &0,12,18,24 & \\
				\cline{2-4}
				&3 &4 & 18,24,42&\\
				\cline{2-4}
				&4 &3 & 12,18,24& \\
				\cline{2-4}
				&5 &2 &12,18 & \\
				\cline{2-4}
				&6 &1 &24 & \\
				\cline{2-4}
				&7 &0 & 42& \\
				\hline
			\end{tabular}
		\end{center}								
	\end{table}
	For any subset $A\subset\mathbb{F}_{2^4}$ of size 1 to 8 with $0\in A$ is obtained by Theorem \ref{zero_added}. Rest of the values for any subsets of $\mathbb{F}_{2^4}$ of size 8 to 16 is obtained by Theorem \ref{compliment}. The spectrum of each subset size in $\mathbb{F}_{2^4}$ is given as Figure \ref{fig:table-24}. The red and green lines are as indicated in Section \ref{F_2_3}.
	\newpage
	\begin{figure}[h]
		\centering
		\includegraphics[width=0.7\linewidth]{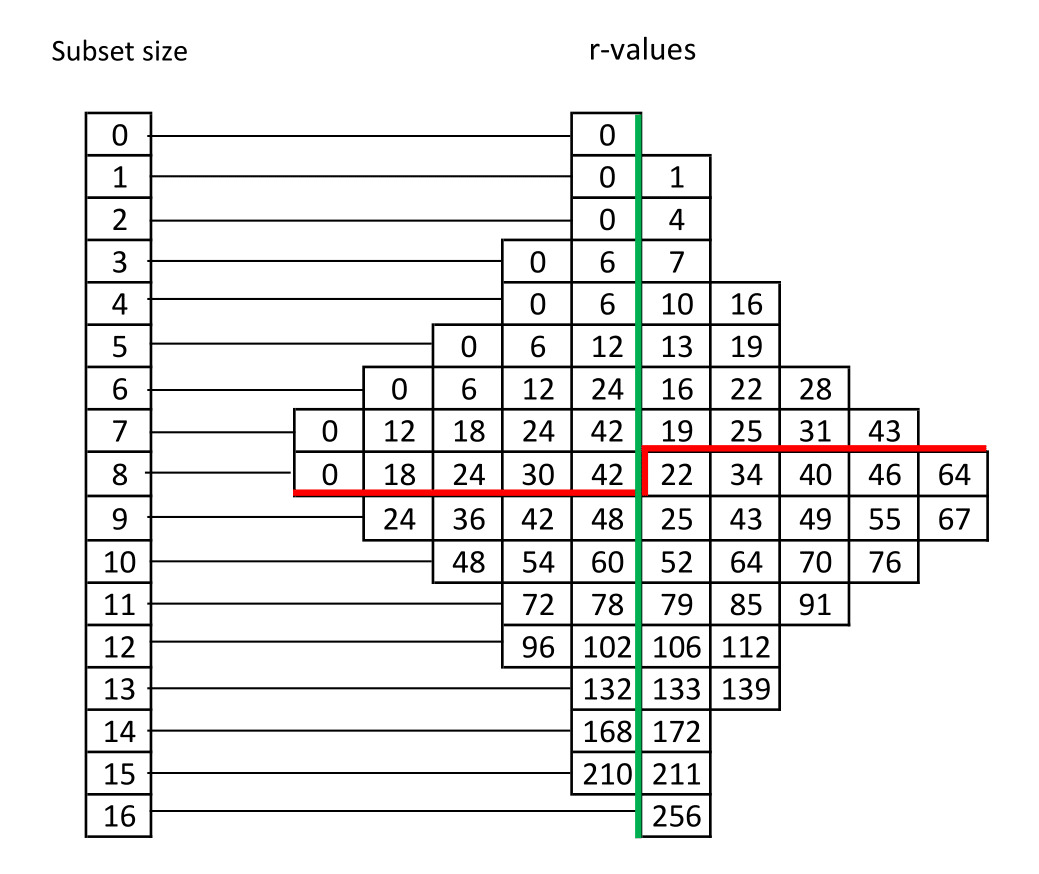}
		\caption{Spectrum of $r$-values in  $\mathbb{F}_{2^4}$}
		\label{fig:table-24}
	\end{figure}
	\subsection{Spectrum of $r$-values in $\mathbb{F}_{2^5}$}\label{F_2_5}		
	It is computationally verified that the above listed result will generate a spectrum of $r$-values in $\mathbb{F}_{2^5}$. In particular the following Figure \ref{fig:table-25} classifies 429,49,67,296 subsets in terms of subset size and $r$-value. The red and green lines are as indicated in Section \ref{F_2_3}.
	\newpage
	\begin{figure}[th]
		\centering
		\includegraphics[width=1\linewidth]{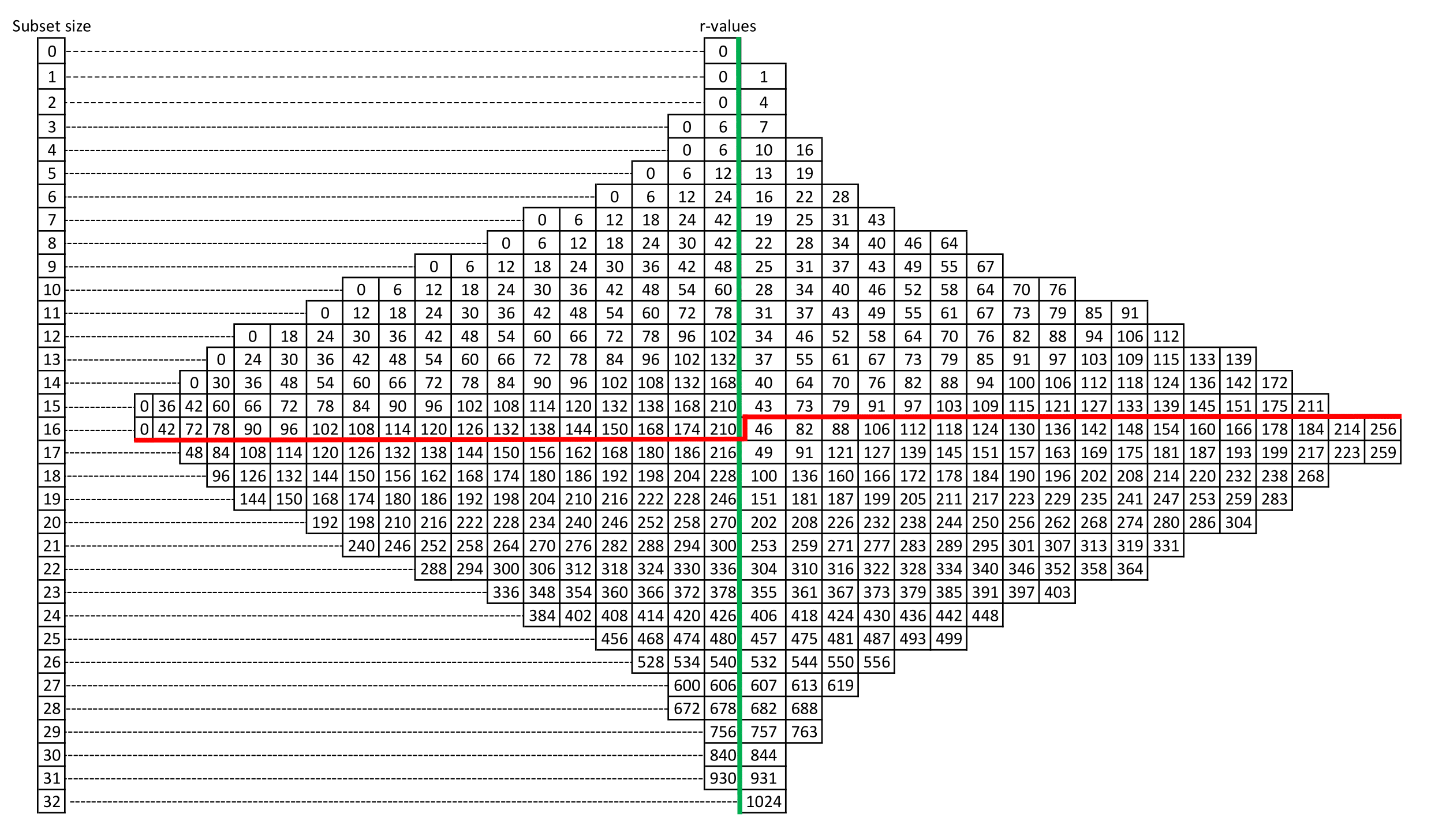}
		\caption{Spectrum of $r$-values in  $\mathbb{F}_{2^5}$}
		\label{fig:table-25}
	\end{figure}
	
	
	\subsection{Enumeration of $r$-values in $\mathbb{F}_{2^n}$, for $n\geq 6$}\label{F_2_n}
	Using the above listed results, it is possible to observe $r$-values of a large number of subsets in $\mathbb{F}_{2^n}$, for $n\geq 6$. The Table \ref{table2-6} gives the some $r$-values for different subset size in $\mathbb{F}_{2^6}\backslash\{0\}$. Using these values we can further continue to observe some collection of $r$-values in $\mathbb{F}_{2^7}$ and so on.
	
					\newpage
	\begin{table}[h]\label{table2-6}	
		\caption{$r$-values of some subsets of $\mathbb{F}_{2^6}\backslash\{0\}$}				
		\begin{center}
			{\scriptsize \begin{tabular}{|c|p{15cm}|}
					\hline
					\textbf{Size} & \textbf{$r$-values} \\ \hline
					0  & 0 \\ \hline
					1  & 0 \\ \hline
					2  & 0 \\ \hline
					3  & 0, 6 \\ \hline
					4  & 0, 6 \\ \hline
					5  & 0, 6, 12 \\ \hline
					6  & 0, 6, 12, 24 \\ \hline
					7  & 0, 6, 12, 18, 24, 42 \\ \hline
					8  & 0, 6, 12, 18, 24, 30, 42 \\ \hline
					9  & 0, 6, 12, 18, 24, 30, 36, 42, 48 \\ \hline
					10 & 0, 6, 12, 18, 24, 30, 36, 42, 48, 54, 60 \\ \hline
					11 & 0, 6, 12, 18, 24, 30, 36, 42, 48, 54, 60, 72, 78 \\ \hline
					12 & 0, 6, 12, 18, 24, 30, 36, 42, 48, 54, 60, 66, 72, 78, 96, 102 \\ \hline
					13 & 0, 6, 12, 18, 24, 30, 36, 42, 48, 54, 60, 66, 72, 78, 84, 96, 102, 132 \\ \hline
					14 & 0, 6, 12, 18, 24, 30, 36, 42, 48, 54, 60, 66, 72, 78, 84, 90, 96, 102, 108, 132, 168 \\ \hline
					15 & 0, 6, 12, 18, 24, 30, 36, 42, 48, 54, 60, 66, 72, 78, 84, 90, 96, 102, 108, 114, 120, 132, 138, 168, 210 \\ \hline
					16 & 0, 6, 12, 18, 24, 30, 36, 42, 48, 54, 60, 66, 72, 78, 84, 90, 96, 102, 108, 114, 120, 126, 132, 138, 144, 150, 168, 174, 210 \\ \hline
					17 & 0, 6, 12, 18, 24, 30, 36, 42, 48, 54, 60, 66, 72, 78, 84, 90, 96, 102, 108, 114, 120, 126, 132, 138, 144, 150, 156, 162, 168, 174, 180, 186, 210, 216 \\ \hline
					18 & 0, 6, 12, 18, 24, 30, 36, 42, 48, 54, 60, 66, 72, 78, 84, 90, 96, 102, 108, 114, 120, 126, 132, 138, 144, 150, 156, 162, 168, 174, 180, 186, 192, 198, 204, 210, 216, 228 \\ \hline
					19 & 0, 12, 18, 24, 30, 36, 42, 48, 54, 60, 66, 72, 78, 84, 90, 96, 102, 108, 114, 120, 126, 132, 138, 144, 150, 156, 162, 168, 174, 180, 186, 192, 198, 204, 210, 216, 222, 228, 246 \\ \hline
					20 & 0, 18, 24, 30, 36, 42, 48, 54, 60, 66, 72, 78, 84, 90, 96, 102, 108, 114, 120, 126, 132, 138, 144, 150, 156, 162, 168, 174, 180, 186, 192, 198, 204, 210, 216, 222, 228, 234, 240, 246, 252, 258, 270 \\ \hline
					21 & 0, 24, 30, 36, 42, 48, 54, 60, 72, 84, 96, 102, 108, 114, 120, 126, 132, 138, 144, 150, 156, 162, 168, 174, 180, 186, 192, 198, 204, 210, 216, 222, 228, 234, 240, 246, 252, 258, 264, 270, 276, 282, 288, 294, 300 \\ \hline
					22 & 0, 30, 36, 42, 48, 54, 60, 72, 96, 108, 120, 126, 132, 138, 144, 150, 156, 162, 168, 174, 180, 186, 192, 198, 204, 210, 216, 222, 228, 234, 240, 246, 252, 258, 264, 270, 276, 282, 288, 294, 300, 306, 312, 318, 324, 330, 336 \\ \hline
					23 & 0, 36, 42, 48, 54, 60, 66, 114, 138, 144, 150, 156, 162, 168, 174, 180, 186, 192, 198, 204, 210, 216, 222, 228, 234, 240, 246, 252, 258, 264, 270, 276, 282, 288, 294, 300, 306, 312, 318, 324, 330, 336, 348, 354, 360, 366, 372, 378 \\ \hline
					24 & 0, 42, 48, 54, 60, 66, 72, 96, 120, 144, 168, 192, 198, 204, 210, 216, 222, 228, 234, 240, 246, 252, 258, 264, 270, 276, 282, 288, 294, 300, 306, 312, 318, 324, 330, 336, 342, 348, 354, 360, 366, 372, 378, 384, 402, 408, 414, 420, 426 \\ \hline
					25 & 0, 48, 54, 60, 66, 72, 144, 168, 192, 216, 240, 246, 252, 258, 264, 270, 276, 282, 288, 294, 300, 306, 312, 318, 324, 330, 336, 342, 348, 354, 360, 366, 372, 378, 384, 402, 408, 414, 420, 426, 456, 468, 474, 480 \\ \hline
					26 & 0, 54, 60, 66, 72, 96, 120, 144, 168, 192, 216, 240, 264, 288, 294, 300, 306, 312, 318, 324, 330, 336, 342, 348, 354, 360, 366, 372, 378, 384, 390, 396, 402, 408, 414, 420, 426, 432, 456, 468, 474, 480, 528, 534, 540 \\ \hline
					27 & 0, 60, 66, 72, 78, 174, 198, 222, 246, 270, 294, 318, 336, 342, 348, 354, 360, 366, 372, 378, 384, 390, 396, 402, 408, 414, 420, 426, 432, 438, 444, 456, 462, 468, 474, 480, 486, 528, 534, 540, 600, 606 \\ \hline
					28 & 0, 66, 72, 78, 120, 144, 192, 216, 240, 264, 288, 312, 336, 360, 384, 402, 408, 414, 420, 426, 432, 438, 444, 450, 456, 462, 468, 474, 480, 486, 492, 498, 528, 534, 540, 546, 600, 606, 672, 678 \\ \hline
					29 & 0, 72, 78, 84, 204, 228, 276, 300, 324, 348, 372, 396, 420, 444, 456, 468, 474, 480, 486, 492, 498, 504, 510, 516, 528, 534, 540, 546, 552, 558, 600, 606, 612, 672, 678, 756 \\ \hline
					30 & 0, 78, 84, 144, 156, 168, 240, 264, 288, 312, 336, 360, 384, 408, 432, 456, 480, 528, 534, 540, 546, 552, 558, 564, 570, 576, 600, 606, 612, 618, 624, 672, 678, 684, 756, 840 \\ \hline
					31 & 0, 84, 90, 156, 162, 168, 216, 222, 228, 234, 240, 258, 330, 354, 378, 402, 426, 450, 474, 498, 522, 546, 570, 600, 612, 618, 624, 630, 636, 642, 672, 684, 690, 696, 756, 762, 840, 930 \\ \hline
					32 & 0, 90, 168, 174, 234, 240, 246, 258, 288, 294, 300, 306, 312, 318, 330, 360, 384, 408, 432, 456, 480, 504, 528, 552, 576, 600, 672, 690, 696, 702, 708, 714, 762, 768, 774, 840, 846, 930 \\ \hline
			\end{tabular}}
		\end{center}
	\end{table}
	\section{Conclusion}
	The study resulted in enumerating $r$-values of large number of subsets of $\mathbb{F}_{2^n}$ by knowing $r$-values in $\mathbb{F}_{2^{n-1}}$. It is observed that every subset $S$ of $\mathbb{F}_{2^n}\backslash\{0\}$ will form a partial Steiner triple system such that the number of blocks in the partial Steiner triple system is equal to $r(A)/6$. For each size $m$, the maximum number of blocks possible in the partial Steiner triple system of order $m$ is equal to the maximum $r$-value in the spectrum of size $m$ in $\mathbb{F}_{2^n}\backslash\{0\}$. These results can be used to study different sets like Sum-free sets and Sidon sets in the future.
		\newpage
        \bibliographystyle{unsrt}

\begin{thebibliography}{99}
	   \bibitem{cameron1990}
Peter Cameron and P.~Erdos, \emph{On the number of sets of integers with
  various properties}, Number Theory (1990).

\bibitem{Claude2023}
Claude Carlet and Stjepan Picek, \emph{On the exponents of apn power functions
  and sidon sets, sum-free sets, and dickson polynomials}, 2023,
  pp.~1507--1525.

\bibitem{Chervak2017}
Ostap Chervak, Oleg Pikhurko, and Katherine Staden, \emph{Minimum number of
  additive tuples in groups of prime order}, The Electronic Journal of
  Combinatorics \textbf{26} (2019).

\bibitem{CILLERUELO20102786}
Javier Cilleruelo, Imre Ruzsa, and Carlos Vinuesa, \emph{Generalized sidon
  sets}, Advances in Mathematics \textbf{225} (2010), no.~5, 2786--2807.

\bibitem{DATSKOVSKY2003193}
Boris~A. Datskovsky, \emph{On the number of monochromatic schur triples},
  Advances in Applied Mathematics \textbf{31} (2003), no.~1, 193--198.

\bibitem{axioms12080724}
Renato~Cordeiro de~Amorim, \emph{On sum-free subsets of abelian groups}, Axioms
  \textbf{12} (2023), no.~8, 724.

\bibitem{Elsholtz}
Christian Elsholtz and Laurence Rackham, \emph{Maximal sum-free sets of integer
  lattice grids}, Journal of the London Mathematical Society \textbf{95}
  (2017), no.~2, 353--372.

\bibitem{Erdo65}
Paul Erdös, \emph{Extremal problems in number theory}, Proc. Sympos. Pure
  Math. \textbf{8} (1965), 181–189.

\bibitem{steiner2020}
Michael~A. Henning and Anders Yeo, \emph{Partial steiner triple systems},
  pp.~171--177, Springer International Publishing, Cham, 2020.

\bibitem{Huczynska2011}
Sophie Huczynska, \emph{Beyond sum-free sets in the natural numbers},
  Electronic Journal of Combinatorics \textbf{21} (2014).

\bibitem{HUCZYNSKA2009831}
Sophie Huczynska, Gary~L. Mullen, and Joseph~L. Yucas, \emph{The extent to
  which subsets are additively closed}, Journal of Combinatorial Theory, Series
  A \textbf{116} (2009), no.~4, 831--843.

\bibitem{LIU2023103614}
Hong Liu, Guanghui Wang, Laurence Wilkes, and Donglei Yang, \emph{Shape of the
  asymptotic maximum sum-free sets in integer lattice grids}, European Journal
  of Combinatorics \textbf{107} (2023), 103614.

\bibitem{MARTIN2006591}
Greg Martin and Kevin O’Bryant, \emph{Constructions of generalized sidon
  sets}, Journal of Combinatorial Theory, Series A \textbf{113} (2006), no.~4,
  591--607.

\bibitem{Mullin}
A.A Mullin, \emph{On mutant sets}, Bulletin of Mathematical Biophysics
  \textbf{24} (1962), 209–215.

\bibitem{Nithish}
Nithish R, Vadiraja G~R, and Prasanna Poojary, \emph{Additive closedness in
  subsets of $\mathbb{Z}_n$}, Communications in Combinatorics and Optimization
  (2024).

\bibitem{sumfreesurvey}
Terence Tao and Van Vu, \emph{Sumfree sets in groups: a survey}, Journal of
  Combinatorics \textbf{8} (2017), 541–552.

\bibitem{TIMMONS2022112659}
Craig Timmons, \emph{Regular saturated graphs and sum-free sets}, Discrete
  Mathematics \textbf{345} (2022), no.~1, 112659.



	\end{thebibliography}
	
\end{document}